\documentclass{amsart}
\usepackage{fullpage}
\usepackage{all2021}
\usepackage{algorithm}
\usepackage{algorithmic}
\newtheorem{example}{Example}
\newtheorem{conjecture}[thm]{Conjecture}
\newcommand{\cA}{\mathcal{A}}
\newcommand{\Log}{\operatorname{Log}}

\newcommand{\R}{\mathbb{R}}
\newcommand{\C}{\mathbb{C}}

\newcommand{\fcc}{\mbox{fcc}}
\newcommand{\ignore}[1]{}
\newcommand{\cS}{\mathbf{S}}
\newcommand{\cT}{\mathbf{T}}
\newcommand{\cU}{\mathbf{U}}
\newcommand{\cP}{\mathbf{P}}
\newcommand{\cQ}{\mathbf{Q}}

\DeclareMathOperator{\adim}{adim}

\usepackage{xcolor}

\begin{document}

\title{The amoeba dimension of a linear space}
\author{Jan Draisma}
\address{Mathematical Institute, University of Bern, Sidlerstrasse 5, 3012 Bern, Switzerland; and Department of Mathematics and Computer Science, Eindhoven University of Technology, P.O.~Box 513, 5600 MB Eindhoven, The Netherlands}
\email{jan.draisma@unibe.ch}

\author{Sarah Eggleston}
\address{Institut f\"ur Mathematik, Albrechtstra{\ss}e 28a, 49076 Osnabr\"uck, Germany}
\email{sarah.eggleston@uni-osnabrueck.de}

\author{Rudi Pendavingh}
\address{Department of Mathmematics and Computer Science, Eindhoven University of Technology, P.O.~Box 513, 5600 MB Eindhoven, The Netherlands}
\email{r.a.pendavingh@tue.nl}

\author{Johannes Rau}
\address{Departamento de Matem\'aticas, Universidad de los Andes,  Carrera 1 \# 18A - 12, 111711 Bogot\'a, Colombia}
\email{j.rau@uniandes.edu.co}

\author{Chi Ho Yuen}
\address{Department of Mathematical Sciences, University of Copenhagen, Universitetsparken 5, 2100 Copenhagen, Denmark}
\email{chy@math.ku.dk}

\thanks{JD was partially supported by Swiss National Science Foundation (SNSF) project grant 200021\_191981 and by Vici grant 639.033.514 from the Netherlands Organisation for Scientific Research (NWO).
JR was supported by the FAPA project ``Matroids in tropical geometry'' from the Facultad de Ciencias, Universidad de los Andes, Colombia.
CHY was supported by the Trond Mohn Foundation project ``Algebraic and Topological Cycles in Complex and Tropical Geometries''; he also acknowledges the support of the Centre for Advanced Study (CAS) in Oslo, Norway, which funded and hosted the Young CAS research project ``Real Structures in Discrete, Algebraic, Symplectic, and Tropical Geometries'' during the 2021/2022 and 2022/2023 academic years}

\begin{abstract}
Given a complex vector subspace $V$ of $\mathbb{C}^n$, the dimension of the amoeba of $V \cap (\mathbb{C}^*)^n$ depends only on the matroid that $V$ defines on the ground set $\{1,\ldots,n\}$. Here we prove that this dimension is given by the minimum of a certain function over all partitions of the ground set, as previously conjectured by Rau.
We also prove that this formula can be evaluated in polynomial time.
\end{abstract}

\maketitle

\section{Introduction} 

\subsection{The goal}
Let $V \subseteq \CC^n$ be a complex vector subspace, and set $X:=V \cap (\CC^*)^n$. We assume throughout that $X \neq \emptyset$, i.e., that $V$ is not contained in any coordinate hyperplane. We denote by $\cA(X) \subseteq \RR^n$ the {\em amoeba} of $X$, defined as the image $\Log(X)$ of $X$
under the map
\[ \Log:(\CC^*)^n \to \RR^n,\ z=(z_1,\ldots,z_n) \mapsto
(\log(|z_1|),\ldots,\log(|z_n|)). \]
An open dense subset of $\cA(X)$ (in the Euclidean topology) is a real
manifold, and we write $\dim_\RR \cA(X)$ for the dimension of this
manifold. We will derive a purely combinatorial expression for
$\dim_\RR \cA(X)$. This expression depends only on the {\em
matroid} $M_V$ that
$V$ defines on the ground set $[n]:=\{1,\ldots,n\}$ (see below), and in
fact makes sense for any matroid $M$ on that ground set, not just those
represented by some complex vector space. Furthermore, we derive an algorithm for computing the said expression that uses a
polynomial number of rank evaluations in $M$. In particular, if $V$ is
given as the row space of an input matrix with, for example, entries
in the field $\QQ(i)$ of Gaussian rationals, then $\dim_\RR \cA(X)$ can be computed in polynomial time in the
bit-length of the input.

\subsection{The matroid of $V$}
For $S \subseteq [n]$ we denote by $r_V(S)$ the dimension
of the image of $V$ under the projection $\CC^n \to \CC^S$. 
More explicitly,
if $d:=\dim_\CC V$ and $V$ is presented as the row space of a $d \times
n$-matrix $A$, then $r_V(S)$ is the dimension of the span of the
columns of $A$ labelled by $S$.  This is the rank function of the 
matroid $M_V$ defined by $V$. The assumption that $X \neq \emptyset$
is equivalent to the condition that $M_V$ has no loops, and
to the condition that $A$ has no zero column.

\subsection{Main results} \label{ssec:Main}

We will prove the following formula, first conjectured by Rau
\cite{Rau20}. Throughout the paper, a {\em partition} of a finite set
$S$ is a set of pairwise disjoint, nonempty subsets of $S$,
called the {\em parts} of the partition, whose union is $S$.

\begin{thm} \label{thm:Main}
We have 
\begin{equation} \label{eq:Main} 
\dim_\RR \cA(X)=\min \sum_{i=1}^k (2 \, r_V(P_i)-1),
\end{equation}
where the minimum is over all partitions
$\{P_1,\ldots,P_k\}$ of $[n]$.
\end{thm}

\begin{example} \label{ex:trivial}
The two obvious upper bounds $2d-1$ (see Lemma~\ref{lm:2dmin1}) and $n$ (the dimension of the ambient space) can be recovered by choosing the trivial partition and the partition into singletons, respectively.
More generally, if the matroid has connected components
$M_1,\ldots,M_k$, then choosing the partition consisting of
their respective ground sets shows that $\dim_\RR \cA(X)
\leq 2d-k$.
\end{example}

The following example, due to Mounir Nisse, shows that the $\dim_\RR \cA(X)$ can drop below $\min\{n,2d-1\}$ even if $M_V$ is connected.

\begin{example}\label{ex:nisse}
Consider the $4\times 7$ matrix
$$
\begin{pmatrix}
1 & 0 & 0 & 0 & * & * & *\\
0 & 1 & 0 & 0 & * & * & *\\
0 & 0 & 1 & 0 & 0 & 0 & *\\
0 & 0 & 0 & 1 & 0 & 0 & *
\end{pmatrix},
$$
where the stars represent any sufficiently general tuple of
complex numbers. 
The matroid represented by such a matrix is connected.
However, the partition of columns
$\{1,2,5,6\},\{3\},\{4\},\{7\}$ yields $\dim_\RR \cA(X)\leq (2\cdot 2-1)+3 \cdot (2\cdot 1-1)=6<\min\{7,2\cdot 4-1\}$.

\end{example}

The next result interprets the expression in the
theorem as the rank of a matroid. To this end, 
let $M$ be a loopless matroid on the finite set $E$, with rank function $r:2^E\rightarrow\mathbb{N}$. For subsets $S_1,\ldots, S_k\subseteq E$, put
$$ \tilde{r}(S_1,\ldots S_k):=\sum_{i=1}^k (2r(S_i)-1)$$ and define the function $r':2^{E}\rightarrow\mathbb{N}$ as
 \begin{equation}\label{rprime}
 r'(S):=\min \tilde{r}(P_1,\ldots, P_k)
 \end{equation}
where the minimum is over all partitions
$\{P_1,\ldots,P_k\}$ of $S$. 

\begin{thm}\label{matroid} Let $M$ be a loopless matroid on
$E$ with rank function $r$, and let $r'$ be defined as in
\eqref{rprime}. Then $r'$ is the rank function of a matroid
$M'$ on $E$.
\end{thm}

Moreover, we will establish the following algorithmic result.

\begin{thm}\label{thm:Algorithm}
The rank function $r'$ of $M'$ can be evaluated on an input set $S
\subseteq E$ by a polynomial number of rank evaluations in
$M$.
\end{thm}

\begin{cor}
There exists an algorithm that, on input of a $d \times n$-matrix $A$ with
entries in $\QQ(i)$ and with no zero column, computes the amoeba
dimension of $V \cap (\CC^*)^n$, where $V$ is the complex row space of
$A$, in polynomial time in the bit-length of $A$.
\end{cor}

\begin{proof}
This is a direct consequence of Theorems~\ref{thm:Main}
and~\ref{thm:Algorithm}, plus the fact that the rank of the submatrix
$A[S]$, obtained from $A$ by picking the columns labelled by any subset
$S \subseteq [n]$, can be computed in polynomial time in the bit-length
of $A$.
\end{proof}

\subsection{A matroid conjecture} \label{subsection:DRYformula}

From \cite{Draisma18d}, we know that, for any irreducible subvariety $X$
of $(\CC^*)^n$, $\dim_\RR \cA(X)$ depends only on the {\em
tropicalisation}
or {\em Bergman fan} $\Trop(X) := \lim_{t \to 0} t \cdot \cA(X)$ of $X$
\cite{Ber-LogarithmicLimitSet}. In our current setting,
where $X$ is the intersection with $(\CC^*)^n$
of a linear space $V \subseteq \CC^n$, $\Trop(X)$ is 
the support $\Sigma(M_V)$ of the \emph{matroid fan} 
of $M_V$ \cite{Stu-SolvingSystemsPolynomial}, 
so it was already known that the amoeba dimension of $X$
depends only on $M_V$. 
However, this dependence via \cite{Draisma18d} is
rather implicit and not quite sufficient for algorithmic computation:
it states that $\dim_\RR \cA(X) = \adim(M_V)$ with  
\begin{equation} \label{eq:DRYformula} 
  \adim(M) := \min\{2\dim_\RR(\Sigma(M)+R)-\dim_\RR(R) \mid R
\subseteq \RR^E \text{ rational subspace}\}.
\end{equation}
Here a {\em rational subspace} of $\RR^E$ is a real subspace spanned by
its intersection with $\QQ^E$, and $\Sigma(M)+R$ is the Minkowski sum of the
$r(E)$-dimensional polyhedral fan $\Sigma(M)$ and the linear
space $R$.
The formula in Theorem~\ref{thm:Main} was inspired by the
following conjecture.

\begin{conjecture} \label{conj:BraidSubspaces}
Let $M$ be a loopless matroid on $E$ and let $\Sigma(M)
\subseteq \RR^E$ be the matroid fan of $M$. 
Then the minimum in \eqref{eq:DRYformula} is attained 
by some space $R$ in the braid arrangement, that is, an 
intersection of a number of hyperplanes of the form $x_i =
x_j$ with distinct $i,j \in E$.
\end{conjecture}

To make the connection with Theorem~\ref{thm:Main}, consider the subspace
$R$ for which $x_i = x_j$ whenever $i$ and $j$ lie in the same part of
the partition $\{P_1,\ldots, P_k\}$.  It is then easy to check that
$\tilde{r}(P_1,\ldots, P_k) = 2\dim_\RR(\Sigma(M)+R)-\dim_\RR(R)$.
Hence clearly $r'(E) \geq \adim(M)$, and 
Conjecture~\ref{conj:BraidSubspaces} is in fact equivalent to the equality
$r'(E) = \adim(M)$.  

In particular, if $M$ is the loopless matroid of a linear space $V
\subseteq \CC^n$, then Conjecture~\ref{conj:BraidSubspaces} follows
from Theorem~\ref{thm:Main} and the formula for the amoeba dimension
from \cite{Draisma18d}. It would be very interesting to have a direct
proof of Conjecture~\ref{conj:BraidSubspaces} that works also for non-representable matroids.

In any case, an important benefit of Theorem~\ref{thm:Main} over the
formula $\dim_\RR \cA(X) = \adim(M_V)$ is that Theorem~\ref{thm:Main}
allows for efficient computation of $\dim_\RR \cA(X)$.

\subsection{Amoeba dimensions and algebraic matroids}

Let $K$ be an infinite field. 
Any irreducible variety $X \subseteq (K^*)^n$ defines an {\em algebraic}
matroid $M_X$ on $[n]$ by declaring $S \subseteq [n]$ to be independent
if the projection from $X$ to $(\CC^*)^S$ is dominant (this is equivalent to
the functions $\{x_i : i \in S\}$, regarded as polynomial functions
on $X$, being algebraically independent over $K$).
If $K = \R$ and $A \subset (\R^*)^n$ is an irreducible semi-algebraic set,
we define the matroid of $A$ as the matroid of its Zariski closure $\overline{A} 
\subset (\R^*)^n$, that is, $M_A := M_{\overline{A}}$.
Note that, for example by the decomposition theorem for semi-algebraic sets, 
$\overline{A} =  (\R^*)^n$ if and only if $A$ contains a nonempty 
Euclidean-open subset. It follows easily that $S \subseteq [n]$ is independent 
in $M_A$ if and only if the projection of $A$
to $\RR^S$ contains a nonempty Euclidean-open subset of $\RR^S$.

\begin{prop} \label{prop:AmoebaMatroid}
Let $X \subset (\C^*)^n$ be an irreducible variety. 
The amoeba $\cA(X)$ defines a matroid $M_{\cA(X)}$, by declaring
$S \subseteq [n]$ to be independent if the projection of $\cA(X)$
to $\RR^S$ contains a nonempty Euclidean-open subset of $\RR^S$.
\end{prop}

\begin{proof}
We consider the ``squared algebraic amoeba'' $\cA'(X) \subset (\R_{>0})^n$
defined as the image of $X$ under the map $|.|^2 \colon (\C^*)^n \to (\R_{>0})^n$, 
$(z_i)_i \mapsto (|z_i|^2)_i$. 
Since $X$ is irreducible and $\CC/\RR$ is a finite separable extension, $X$
is also irreducible with respect to the real Zariski topology of 
$(\C^*)^n$. Since $|.|^2$ is a real regular map, it
follows that $\cA'(X)$ is an irreducible semi-algebraic set with associated
(algebraic) matroid $M_{\cA'(X)}$.
Taking square roots and logarithm provides a diffeomorphism 
$(\R_{>0})^n \to \R^n$ sending $\cA'(X)$ to $\cA(X)$, and this 
diffeomorphism is compatible with projections in the obvious sense. 
Hence $M_{\cA(X)}$ as defined in the statement is equal to $M_{\cA'(X)}$,
which proves the claim. 
\end{proof}

\begin{prop}
If $X=V \cap(\CC^*)^n$ where $V \subseteq \CC^n$ is a linear space
not contained in any coordinate hyperplane, then the matroid $M_{\cA(X)}$
is the matroid from Theorem~\ref{matroid} with the rank function $r'$
constructed from $M:=M_X=M_V$. In particular, $r'$ defines an algebraic matroid. 
\end{prop}

\begin{proof}
Let $S$ be a subset of $[n]$ and let $X_S$ be the projection of $X$
into $(\CC^*)^S$. Then $S$ is independent in $M_{\cA(X)}$ if and only if
$\dim \cA(X_S)=|S|$, which by Theorem~\ref{thm:Main} applied to $X_S$
is equivalent to the condition that $r'(S)=|S|$.
\end{proof}

We note that the proof above only uses Theorem~\ref{thm:Main} and
not Theorem~\ref{matroid}. Hence when $M=M_V$ is a matroid representable
over $\CC$, 
Theorem~\ref{thm:Main} and the two propositions above {\em imply}
Theorem~\ref{matroid}. 

Given that, for $X$ arising from a linear space, $M_{\cA(X)}$ has such
a nice description in terms of $M_X$, one might wonder whether the same
holds for general irreducible varieties $X$. In particular, one could ask
whether the amoeba dimension of $X$ is also determined by $M_X$, for
instance via the formula from Theorem~\ref{thm:Main}. 

The answer, however, is no in general. For an extreme counterexample,
let $X$ be a $d$-dimensional subtorus of $(\CC^*)^n$ such that $M_X$
is the uniform matroid $U_{d,n}$ of rank $d$ on $[n]$---this is achieved
by choosing the subtorus $X$ whose Lie algebra, i.e., tangent space at
$(1,\ldots,1)$, is the $\CC$-span $\langle R \rangle_\CC$ of any rational
subspace $R$ of $\QQ^n$ that represents the matroid $U_{d,n}$. Since
$X$ is a subtorus, its amoeba is a linear space, namely, the real
span $\langle R \rangle_\RR$. Hence $\dim_\RR \cA(X)$ is $d$, which is
the minimum among all amoeba dimensions of $d$-dimensional varieties
(actually, this minimum is attained only for translates of subtori;
see \cite[Theorem 4.6]{Nisse22}).  On the other hand, the tangent
space $V$ to $X$ at any point $p \in X$ equals $p \cdot \langle R
\rangle_\CC$ (the Hadamard product) and hence also represents the matroid
$U_{d,n}$. Therefore $\dim_\RR \cA(V \cap (\CC^*)^n)=\min\{n,2d-1\}$, as is
easily seen with the formula in Theorem~\ref{thm:Main}.

\subsection{Organisation of this paper}

In Section~\ref{sec:Proof1}, we prove Theorem~\ref{thm:Main}, in
Section~\ref{sec:AmoebaMatroid} we prove Theorem~\ref{matroid},
and in Section~\ref{sec:Algorithm} we derive the algorithm in
Theorem~\ref{thm:Algorithm}.

\subsection*{Acknowledgments}

This paper grew out of several sources: JR's talk on amoebas \cite{Rau20} where the formula of Theorem~\ref{thm:Main} was first conjectured, SE's Master's thesis \cite{Eggleston22} at the University of Bern under the supervision of JD, work by CHY on a combinatorial analysis of the Jacobian of $\Log$ at a general point of a linear space, and RP's work on the matroid $M'$ of Theorems~\ref{matroid} and~\ref{thm:Algorithm}.
CHY thanks Kris Shaw for suggesting this problem to him.
All authors thank Frank Sottile for discussions on an early version of this work.

\section{Proof of Theorem~\ref{thm:Main}}
\label{sec:Proof1}

\subsection{Proof of the inequality $\leq$ in
\eqref{eq:Main}}
\label{ssec:ineq1}

As mentioned in \S\ref{subsection:DRYformula}, 
the inequality $\leq$ in
\eqref{eq:Main} is a consequence of 
\cite{Draisma18d} via 
$\dim_\RR \cA(X) = \adim(M_V) \leq r'(E)$.
Nevertheless, we include a short proof which also 
serves as preparation for later arguments. 

Recall from \cite{Draisma18d} that if a closed subvariety $X$ of $(\CC^*)^n$
is stable under a subtorus $T$ of $(\CC^*)^n$, and if we set $Y:=X/T
\subseteq (\CC^*)^n/T \cong (\CC^*)^m$ with $m=n-\dim_\CC(T)$, then we have
a surjective map $\cA(X) \to \cA(Y)$ whose fibres are translates of
$\cA(T)$. It then follows that
\[ \dim_\RR \cA(X)=\dim_\RR \cA(Y) + \dim_\RR
\cA(T)=\dim_\RR \cA(Y) + \dim_\CC T. \]

\begin{lm} \label{lm:2dmin1}
Let $X=V \cap (\CC^*)^n$ where $V$ is a $d$-dimensional complex
subspace of $\CC^n$ not contained in any coordinate
hyperplane. Then $\dim_\RR \cA(X) \leq 2d-1$. 
\end{lm}

\begin{proof}
Since $V$ is closed under scalar multiplication, $X$ is stable under
the one-dimensional torus $T=\{(t,\ldots,t) \mid t \in \CC^*\}$.
Then $Y:=X/T$ has dimension $d-1$, and hence $\dim_\RR
\cA(Y) \leq \dim_\CC Y = 2(d-1)$. It therefore follows from the above that 
\[ \dim_\RR \cA(X)=\dim_\RR \cA(Y) + \dim_\CC T \leq
2(d-1)+1=2d-1, \]
as desired.
\end{proof}

\begin{prop}
In Theorem~\ref{thm:Main}, the inequality $\leq$ holds. 
\end{prop}

\begin{proof}
Let $\{P_1,\ldots,P_k\}$ be a partition of $[n]$.
Let $V_i$ be the image of $V$ under the projection $\CC^n \to
\CC^{P_i}$, set $d_i:=\dim_\CC V_i$, and let $X_i$ be the intersection
of $V_i$ with $(\CC^*)^{P_i}$. Since $V$ is contained in $\prod_i V_i$,
$X$ is contained in $X':=\prod_i X_i$, and hence also
$\cA(X) \subseteq \cA(X')$. Then we find
\[ \dim_\RR \cA(X) \leq \dim_\RR \cA(X') = \sum_i \dim_\RR
\cA(X_i) \leq \sum_i (2d_i-1) \]
where the equality follows from the fact that the amoeba of a product
is the product of the amoebas, and where the last inequality follows
from Lemma~\ref{lm:2dmin1} applied to each $V_i$.
\end{proof}

\subsection{Proof of the inequality $\geq$ in
\eqref{eq:Main}}

To prove $\geq$ in Theorem~\ref{thm:Main}, we need to construct a
partition for which equality holds. We will do so in an inductive manner.

By definition, $m:=\dim_\RR \cA(X)$ equals the maximum, over all $p
\in X$, of the real rank of the real linear map $d_p \Log: T_p X=V \to
\RR^n$, and this maximum is attained in an open dense subset of $X$
(in the Euclidean topology).

In what follows, for vectors $v,w \in \CC^n$, we write $v \cdot w$ for
their Hadamard product $(v_1w_1,\ldots,v_nw_n)$; and if $w \in (\CC^*)^n$,
then we write $v/w$ for the Hadamard quotient
$(v_1/w_1,\ldots,v_n/w_n)$. We also write $\one \in \CC^n$
for the all-one vector $(1,\ldots,1)$. Furthermore, if $z$
is a complex number or vector of complex numbers, then we
write $\Re(z)$ and $\Im(z)$ for the real and imaginary
parts of parts of $z$, respectively; we also use this
notation for a subset of vectors in $\CC^n$ (e.g., a subspace).

\begin{lm} \label{lm:dpLog}
For $p \in X$ and $v \in T_p X=V$ we have 
\[  (d_p \Log)(v)=\Re(v/p)=(\Re(v_1/p_1),\ldots,\Re(v_n/p_n)). \]
\end{lm}

\begin{proof}
Define $w:=v/p$ and decompose $w=x + iy$ with $x,y \in \RR^n$. 
For $\epsilon \in \RR$ tending to $0$ we have
\[ \Log(p+\epsilon v)=\Log(p(\one+\epsilon
w))=\Log(p)+\Log(\one+\epsilon (x+iy))= \Log(p)+\epsilon x +
O(\epsilon^2). \]
The last equality holds because in each component,
\begin{align*}
    \log|1 + \epsilon(x_j + iy_j)| &= \log\left( \sqrt{1 + 2\epsilon x_j + \epsilon^2(x_j^2 + y_j^2)}\right) = \frac{1}{2}\log\left( 1 + 2\epsilon x_j + \epsilon^2(x_j^2 + y_j^2)\right) \\
    &= \frac{1}{2}(2\epsilon x_j + \epsilon^2(x_j^2 + y_j^2) - \frac{1}{2}\left(2\epsilon x_j + \epsilon^2(x_j^2 + y_j^2)\right)^2 + \dots)=\epsilon x_j + O(\epsilon^2).
\end{align*}
This implies $(d_p \Log)(v)= x=\Re(v/p)$, as desired.
\end{proof}

Pick $p_0 \in X$ such that the linear map $A:=d_{p_0} \Log:V
\to \RR^n$ has the maximal possible rank $m$. This means that $A$ is a point in the real manifold $\Hom_\RR(V,\RR^n)_m$ of real linear maps $V \to \RR^n$ of rank precisely $m$. We will use the following description of the tangent space $T_A
\Hom_\RR(V,\RR^n)_m$ (see e.g.~\cite[Example 14.16]{Harris1992}):
\begin{equation} \label{eq:Tangent} T_A \Hom_\RR(V,\RR^n)_m=\{B \in \Hom_\RR(V,\RR^n) \mid B\ker(A)
\subseteq \im(A) \}. \end{equation}
Now, for a Euclidean-open neighbourhood $U$ of $p_0$ in $X$,
the map 
\[ \Phi:U \to \Hom_\RR(V,\RR^n)_m,\quad p \mapsto d_p \Log \]
is a map of smooth manifolds. 

\begin{lm} \label{lm:dpphi}
The derivative $d_{p_0} \Phi$ is the map 
\[ T_{p_0} U=V \to T_{A} \Hom_\RR(V,\RR^n)_m, \quad u \mapsto (v \mapsto - \Re((v
\cdot u)/(p_0 \cdot p_0))) \]
where $A=d_{p_0} \Log \in \Hom_\RR(V,\RR^n)$.
\end{lm}

\begin{proof}
For $u \in V$ define $w:=u/p_0$. 
Then, for $\epsilon$ real and tending to zero, and for $v
\in V$, by Lemma~\ref{lm:dpLog} we have 
\begin{align*} 
\Phi(p_0+\epsilon u)(v)&=
(d_{p_0+\epsilon u} \Log)(v)=\Re(v/(p_0+\epsilon u))
=\Re(v/(p_0(\one+\epsilon w)))\\
&=\Re((v/p_0)\cdot(\one-\epsilon w)) + O(\epsilon^2) 
=\Re(v/p_0) - \epsilon \Re((v\cdot w)/p_0) + O(\epsilon^2). \end{align*}
The coefficient of $\epsilon$ is the expression in the
lemma. 
\end{proof}

We simplify the situation as follows: we replace $V$ by $V/p_0$, $X$ by
$X/p_0$, and $p_0$ by $\one$. This only translates the amoeba of $X$,
and it has no effect on the matroid $M_V$.  Consequently, both sides in
\eqref{eq:Main} are unaltered.
In this simplified setting, Lemma~\ref{lm:dpLog} says that
\[ Av=(d_\one \Log)(v)=\Re(v). \]
In particular, 
\begin{align}  \label{eq:dim}
m&=\dim_\RR \cA(X)=\dim_\RR V - \dim_\RR \ker(A)=\dim_\RR
V - \dim_\RR (V \cap (i\RR^n))\\ &= \dim_\RR V - \dim_\RR (V
\cap \RR^n). \notag
\end{align}
Furthermore, Lemma~\ref{lm:dpphi} says that $d_\one \Phi$ is the linear
map that sends $u$ to the linear map $B_u(v):=-\Re(v \cdot u)$. We now
come to the crucial point in the proof.

\begin{prop} \label{prop:Hadamard}
Under the standing assumption that $p=\one \in V$ is a point where
$d_p\Log$ has the maximal rank $m=\dim_\RR \cA(X)$, the real vector space $V+\RR^n
\subseteq \CC^n$ is closed under Hadamard multiplication with the real
vector space $V \cap \RR^n$.
\end{prop}

\begin{proof}
For each $u \in V$, the linear map $B_u:V \to \RR^n$ lies in
the image of $d_\one \Phi$ and hence in $T_A
\Hom_\RR(V,\RR^n)_m$. By \eqref{eq:Tangent}, this means that 
$B_u$ maps $\ker(A)$ into $\im(A)$.  Since $A$
maps a vector to its real part, we have $\ker(A)=V \cap
(i\RR^n)$ and $\im(A)=\Re(V)$. Hence we find that $B_u(v)=-\Re(v
\cdot u)$ is in $\Re(V)$ for all $v \in V \cap (i
\RR^n)$. Then, for $v \in V \cap \RR^n$, we have $i v \in (V
\cap (i\RR^n))$ and hence the real part of $-iv \cdot u$ is in
$\Re(V)$; but this is also the imaginary part of $v \cdot
u$. Furthermore, since $\Re(V)$ is also the set of
imaginary parts of vectors in $V$ (recall that $V$ is a complex vector 
space, so $V$ is invariant under the multiplication by $i$), we
find that for all $u \in V$ and all $v \in V \cap \RR^n$,
the imaginary part of $v \cdot u$ equals that of a vector in
$V$, so that $v \cdot u \in V+\RR^n$. Thus 
\[ (V \cap \RR^n) \cdot V \subseteq V+\RR^n, \]
and since clearly also 
\[ (V \cap \RR^n) \cdot \RR^n \subseteq \RR^n \subseteq
V+\RR^n \]
the proposition follows. 
\end{proof}

\begin{prop} \label{prop:Split}
Assume that $m=\dim_\RR \cA(X)< 2d-1$. Then there exists a partition
$\{P_1, P_2\}$ of $[n]$ into two parts with the following
property. Let $V_i$ be the projection of $V$ in $\CC^{P_i}$
and set $X_i:=V_i \cap (\CC^*)^{P_i}$. Then
\[ \dim_\RR \cA(X)=\dim_\RR \cA(X_1) + \dim_\RR \cA(X_2). \]
\end{prop}

\begin{proof}
In this case, $\ker d_{\one} \Log=V \cap (i \RR^n)$ has real dimension at least $2$. Hence so does $V \cap \RR^n$. Let $v \in V \cap \RR^n$ be a vector linearly independent from $\one$. After adding a suitable multiple of $\one$, we may assume that all entries of $v$ are positive, and after scaling we may assume that the maximal entry of $v$ equals $1$. Let $P_1 \subseteq [n]$ be the positions where $v$ takes this maximal value $1$, and let $P_2$ be the complement of $P_1$ in $[n]$. By construction, $P_1$ and $P_2$ are both nonempty.

By Proposition~\ref{prop:Hadamard}, $V+\RR^n$ is preserved under Hadamard
multiplication with $v$. Iterating this multiplication and taking the
limit, we find that $V+\RR^n$ is preserved under setting the coordinates
labelled by $P_2$ to zero. 
Then $V+\RR^n$ is also preserved under
setting the coordinates labelled by $P_1$ to zero.  Hence we
have 
\[ V+\RR^n = (V_1 + \RR^{P_1}) \times (V_2 +
\RR^{P_2}),\]
where $V_1,V_2$ are defined in the proposition.

We assume that for $i=1,2$, the all-one vector $\one_i \in V_i \subseteq
\CC^{P_i}$ has the same property required of $\one$, namely,
that $\dim_\RR \cA(X_i)$ equals the rank of the linear map $d_{\one_i}
\Log:V_i \to \RR^{P_i}$. (This might not follow from the corresponding
property of $\one$, but it may be achieved by picking the original $p_0$
in a suitable dense subset of $X$ and then dividing by that $p_0$.)

Now we have 
\begin{align*} \dim_\RR \cA(X)
&= \dim_\RR V - \dim_\RR (V \cap \RR^n)
= \dim_\RR (V+\RR^n)- \dim_\RR \RR^n\\
&= \dim_\RR (V_1 + \RR^{P_1}) + \dim_\RR (V_2+ \RR^{P_2}) -
|P_1|-|P_2|\\
&= (\dim_\RR (V_1 + \RR^{P_1})-\dim_\RR \RR^{P_1})+
(\dim_\RR (V_2 + \RR^{P_2})-\dim_\RR \RR^{P_2})\\
&= (\dim_\RR V_1 - \dim_\RR (V_1 \cap \RR^{P_1}))
+ (\dim_\RR V_2 - \dim_\RR (V_2 \cap \RR^{P_2}))\\
&= \dim_\RR \cA(X_1) + \dim_\RR \cA(X_2),
\end{align*}
as desired. Here we use the dimension formula
for vector subspaces and \eqref{eq:dim} three times
for $V,V_1,$ and $V_2$.
\end{proof}

\begin{proof}[Proof of Theorem~\ref{thm:Main}.]
The inequality $\leq$ was proved in \S\ref{ssec:ineq1}. For $\geq$
we proceed by induction on $n$; we therefore assume that the
inequality holds for all strictly smaller values of $n$. 

Now if $\dim_\RR \cA(X)=2d-1$, where $d=\dim_\CC V$, then $\geq$ is
witnessed by the partition of $[n]$ into a single part $P_1$.  Otherwise,
by Proposition~\ref{prop:Split}, there is a partition
$\{P_1,P_2\}$ of $[n]$ such that
\[\dim_\RR \cA(X)=\dim_\RR \cA(X_1) + \dim_\RR \cA(X_2) \]
where $X_i:=V_i \cap (\CC^*)^n$ and $V_i$ is the projection
of $V$ onto $\CC^{P_i}$. Since $P_1$ and $P_2$ both have
cardinalities strictly smaller than $n$, the induction
hypothesis applies: there exist partitions
$\{P_{i1}, \ldots, P_{ik_i}\}$ of $P_i$ for $i=1,2$ with 
\[\dim_\RR \cA(X_i)=\sum_{j=1}^{k_i} (2 \, r_{V_i} (P_{ij}) -1
). \]
Then the partition $\{P_{11},\ldots,P_{1k_1},
P_{21}, \ldots, P_{2k_2}\}$ of $[n]$ has the desired
property for $X$.
\end{proof}

\section{The function $r'$ is a rank function} \label{sec:AmoebaMatroid}

In this section, we analyse the right-hand side of the amoeba dimension
formula in Theorem~\ref{thm:Main}, and show that it is the rank function
of another matroid (Theorem~\ref{matroid}). We fix a finite
set $E$.

\subsection{Preliminaries on multisets of subsets}

We study finite multisets of subsets of $E$ and
denote these by boldface letters such as $\cS$. If $\cS,\cT$ are such
multisets, then so is their multiset union $\cS \cup \cT$.
We write $\# \cS$ for the number of elements of $\cS$
counting multiplicities.

Denote by $\succeq$ the transitive relation on
the set of finite multisets of subsets of $E$ such that $\cS\succeq\cT$
if and only if there is a sequence of multisets
$$\cS=\cS_1,\ldots,\cS_\ell=\cT$$ 
so that each $\cS_{i+1}$ arises from $\cS_i$ by replacing some {\em
intersecting pair} $S, S' \in \cS_i$, i.e., a pair such that $S \cap S'
\neq \emptyset$, with the pair $S\cap S', S\cup S'$. Note that we have
$\emptyset\not \in \cS\Rightarrow \emptyset\not \in \cT$. Furthermore,
if we define $n(\cS):=\sum_{S\in \cS} |S|^2$, then we have $n(\cS_{i+1})
\geq n(\cS_i)$ with equality if and only if $\cS_{i+1}=\cS_i$. Hence

$$\cS\succeq\cT\Longrightarrow n(\cS) \leq n(\cT),$$ 
with equality if and only if $\cS=\cT$. It
follows that $\succeq$ is a partial order on finite multisets of
subsets of $E$. Also, since a multiset $\cS$ with $k$ elements (counting
multiplicities) has $n(\cS)\leq k |E|^2$, there are no infinite decreasing sequences in the partial order $\succeq$. Hence, for any multiset $\cS$, there exists $\cT \preceq \cS$ with $\cT$ minimal with respect to $\preceq$, that is, $\cT \succeq \cU$ implies $\cT= \cU$. Moreover $\cT$ is minimal with respect to $\preceq$ if and only if it is {\em cross-free}: there are no sets $T, T'\in \cT$ that {\em cross} in the sense that $T\cap T', T\setminus T', T'\setminus T$ are all nonempty. 

For any multiset $\cS$ of subsets of $E$, the {\em finest common
coarsening} is the set of subsets of $E$ defined by 
$$\fcc(\cS):= \left\{ T\subseteq \bigcup
 \cS \mid T\text{ is an inclusion-wise minimal nonempty
 set so that } \forall S \in \cS: S\cap
 T=\emptyset\text{ or } S\subseteq T\right\}.$$
Intuitively, every $T\in\fcc(\cS)$ is obtained from starting with some nonempty $S\in\cS$, 
and keep merging with any other $S'\in\cS$ with nonempty intersection until it stabilises.
In particular, $\fcc(\cS)$ is a partition of $\bigcup \cS$, and
the definition of the partial order
$\succeq$ implies that $$\cS\succeq\cT\Longrightarrow\fcc(\cS)=
  \fcc(\cT).$$
Note that if $\cT$ is cross-free, then $\fcc(\cT)\subseteq \cT$.

For a partition $\cP$ of $S \subseteq E$ and a partition
$\cP'$ of $S' \subseteq E$, we define 
$$\cP\vee \cP':= \fcc(\cP\cup \cP')\text{ and }\cP\wedge \cP':=  \{P\cap
 P': P\in \cP, P'\in \cP'\}\setminus\{\emptyset\}.$$
Then $\cP\vee \cP'$ is a partition of $S\cup S'$, and $\cP\wedge \cP'$
is a partition of $S\cap S'$.

\subsection{A new matroid from $M$} 

Now let $M$ be a loopless matroid on $E$ with rank function
$r:2^E\rightarrow\mathbb{N}$. Recall from \S\ref{ssec:Main}
that, for a multiset $\cS$ consisting of nonempty subsets
$S_1,\ldots,S_k$ of $E$, we have defined 
$$ \tilde{r}(\cS)=\tilde{r}(S_1,\ldots S_k):=\sum_{i=1}^k
2r(S_i)-1.$$ 
Furthermore, we have defined $r':2^{E}\rightarrow\mathbb{N}$ as
\[
r'(S):=\min\left\{\tilde{r}(\cP) \mid \cP
\text{ is a partition of } S \right\}. \]
A partition $\cP$ of $S$ with $r'(S)=\tilde{r}(\cP)$ will be called an
{\em optimal partition of $S$} or simply {\em optimal for
$S$}. 

The submodularity of $r$ implies that, if
$\cS=\cS_1,\ldots,\cS_\ell=\cT$ is a sequence witnessing
$\cS \succeq \cT$ and $\cS_{i+1}$ arises from $\cS_i$ by
replacing the pair $S,S'$ by $S \cup S',S \cap S'$, then 
$$\tilde{r}(\cS_i)-\tilde{r}(\cS_{i+1})=(2r(S)-1)+ (2r(S')-1)- (2r(S\cup S')-1)-(2r(S\cap S')-1)\geq 0$$
for each $i$, so that 
$$\cS\succeq\cT\Longrightarrow
\tilde{r}(\cS)\geq \tilde{r}(\cT).$$

\begin{lm} \label{rank}Let $\cP, \cP'$
be partitions of $S, S'\subseteq E$, respectively. There is a
partition $\cQ$ of $S\cap S'$ so that $$\tilde{r}(\cP)+\tilde{r}(\cP')\geq \tilde{r}(\cP\vee \cP')+ \tilde{r}(\cQ),$$
 $\cQ$ is a coarsening of $\cP\wedge \cP'$, and $\#\cP+\#\cP'=\#\cP\vee\cP'+\#\cQ$.
\end{lm}

In the proof, we use the notation $c(\cS)_e$ for the number of sets $S$
in the multiset $\cS$ that contain a given $e\in E$ (counting multiple
occurrences of $S$). Note that $$\cS\succeq\cT\Longrightarrow c(\cS)_e=
c(\cT)_e.$$

\proof Set $\cS:=\cP \cup \cP'$ and let $\cT$ be a cross-free multiset so that $\cS\succeq \cT$. We have 
$$\cP\vee \cP'= \fcc(\cP\cup \cP')=\fcc(\cS)=\fcc(\cT),$$
Since $\cT$ is cross-free, we have $\fcc(\cT)\subseteq \cT$. Let $\cQ$ be the multiset that arises from $\cT$ by taking away $\fcc(\cT)$. 

Then, for each $e \in E$, 
$$c(\cQ)_e=c(\cT)_e-c(\fcc(\cT))_e=c(\cS)_e-c(\fcc(\cS))_e=\left\{\begin{array}{ll} 
1-1=0&\text{if }e\in S\cup S'\text{ and }e\not\in S\cap S'\\
2-1=1&\text{if }e\in S\cap S'\\
0-0=0&\text{if }e\not\in S\cup S'
\end{array}\right.$$
So $\cQ$ is a partition of $S\cap S'$, and since each element of $\cT$ arises by taking unions and intersections starting from $\cP\cup\cP'$, $\cQ$ is a coarsening of $\cP\wedge \cP'$. Since $\cS\succeq \cT$, we have 
$$\tilde{r}(\cP)+\tilde{r}(\cP')=\tilde{r}(\cS)\geq  \tilde{r}(\cT)= \tilde{r}(\cP\vee \cP')+ \tilde{r}(\cQ)$$
as required. Finally, $\#\cP+\#\cP'=\#\cS=\#\cT=\#(\cP\vee \cP')+ \#\cQ.$
\endproof

\subsection{Proof of Theorem~\ref{matroid}}

\proof We show that $r'$
satisfies the matroid rank axioms:
 
\begin{enumerate}
\item 
$r'(S)\geq 0$ for all $S\subseteq E$: Let $\cP$ be an
optimal partition of $S$. As $M$ is loopless, we have $r(P)\geq 1$ whenever $P\neq \emptyset$, and so
$$r'(S)=\tilde{r}(\cP)=\sum_{P\in \cP} (2r(P)-1)\geq 0,$$ 
as required.

\item $r'(S) \leq |S|$ for all $S \subseteq E$: this follows by taking the partition $\cP$ to be the partition into singletons.

\item 
$r'(S')\leq r'(S)$ whenever $S'\subseteq S\subseteq E$: Let
$\cP$ be an optimal partition of $S$.
Then $$r'(S')\leq \tilde{r}(\cP\wedge\{S'\})\leq \tilde{r}(\cP) = r'(S)$$
as required. 

\item 
$r'(S)+r'(S')\geq r'(S\cup S')+r'(S\cap S')$ for all $S, S'\subseteq E$:
Let $\cP, \cP'$ be optimal partitions of $S, S'$,
respectively. By Lemma \ref{rank}, there is a partition $\cQ$
of $S\cap S'$ so that
$$r'(S)+r'(S')=\tilde{r}(\cP)+\tilde{r}(\cP')\geq \tilde{r}(\cP\vee \cP')+ \tilde{r}(\cQ)\geq r'(S\cup S')+r'(S\cap S'),$$
as required.\qedhere
\end{enumerate}
\endproof

\subsection{Structure of $M'$}

The matroid on $E$ with rank function $r'$ is denoted $M'$. 
We make a few observations on the structure of $M'$ in relation to $M$.

If $M$ and $N$ are matroids on a ground set $E$, then $M$ is
a {\em quotient} of $N$ if each flat of $M$ is also a flat
of $N$. 
\begin{lem} 
  The matroid $M$ is a quotient of $M'$.
\end{lem}

\begin{proof}
  Let $F$ be a flat of $M$. Assume that $F$ is not a flat in $M'$. Then
	there exists $e \in E \setminus F$
	such that $r'(F +e) = r'(F)$. Let $\{P_1,
	\dots, P_k\}$ 
	be an optimal partition of $F +e$. We may assume $e \in P_1$.
	Now $\tilde{r}(P_1, \dots, P_k) = r'(F)$ implies that 
	$r(P_1) = r(P_1 -e)$. In particular, $e$ is contained 
	in any flat containing $P_1 -e \subset F$, a contradiction
	to $F$ being a flat in $M$. Hence the claim follows. 
\end{proof}
The {\em truncation} of a matroid $M$ of rank $d>0$ is the matroid $N$ on the same ground set with rank function $r_N(S):=\min\{r_M(S), d-1\}$.
\begin{lem} \label{lem:truncation}If $M$ has rank $d>1$ and $N$ is the truncation of $M$, then $$r'_N(S)=\min\{r'_M(S), 2d-3\}$$ for all $S\subseteq E$.
\end{lem}
\begin{proof}
	Clearly $r'_N(S)\leq r'_M(S)$. Suppose $r'_N(S)< r'_M(S)$. Let $	
	\cP$ be an optimal partition of $S$ with respect to $N$. Then 
	$$\sum_{P\in \cP} (2r_N(P)-1) = \tilde{r}_N(\cP)=
	r'_N(S)< r'_M(S)\leq  \tilde{r}_M(\cP)=\sum_{P\in
	\cP} (2r_M(P)-1)$$
	hence $r_M(P)> r_N(P)=\min\{r_M(P), d-1\}$ for some $P\in \cP$. Then $r_M(P)=d$ and 
	$r_N(P)=d-1$, so that $$2d-3 \geq 2r_N(S)-1= \tilde{r}_N(\{S\})\geq r'_N(S)=\tilde{r}_N(\cP)\geq 2r_N(P)-1 = 2d-3.$$
	Then $r'_N(S)=2d-3$, as required.
\end{proof}

Let $c\in\mathbb{N}$. A matroid $M$ on the ground set $E$ and with rank function $r$ is {\em $c$-connected} if there is no subset $S\subseteq E$ so that $|S|\geq c, |E\setminus S|\geq c$, and 
$$r(S)+r(E\setminus S)-r(E) < c.$$
The next example illustrates that even if we assume that $M$ has a high connectivity (and representable over $\mathbb{C}$), the rank of $M'$ may still be strictly less than the trivial bound $\min\{|E|,  2r(M)-1\}$ from 
Example~\ref{ex:trivial}. 
Indeed, a high connectivity will not even force the existence of an optimal partition into few parts.

\begin{example} Let $c,k$ be positive integers so that $k>2c+1$. Consider the matroid $M$ that arises from the disjoint union of matroids $M_1, \ldots, M_k$ by truncating $c$ times, where each $M_i$ has ground set $E_i$ and is isomorphic to the uniform matroid $U_{c,2c}$. Let $E$ be the ground set of $M$, and let $\cP_0:=\{E_1,\ldots, E_k\}$.  Then the ground set of $M$ has $n:=|E|=2ck$ elements, the rank of $M$ is $d:=ck-c$, and the rank in $M$ of $S\subseteq E$ equals 
$$r(S)=\min\left\{d, \sum_{i=1}^k \min\{c, |S\cap E_i|\}\right\}.$$

We first verify that $M$ is $c$-connected. Let $S\subseteq E$ with $|S|\geq c$ and $|E\setminus S|\geq c$. If $r(E\setminus S)=d$, then 
$r(S)+r(E\setminus S)-r(E) = r(S)\geq \min\{c,|S|\}=c$. If $r(S)=d$, then similarly
$r(S)+r(E\setminus S)-r(E) \geq c$. In the remaining case $r(S)<d$ and $r(E\setminus S)<d$, hence 
$$r(S)+r(E\setminus S)-r(E)=\sum_{i=1}^k \left(\min\{c,|S\cap E_i|\}+ \min\{c,|S\setminus E_i|\}\right)-(ck-c)\geq ck-(ck-c)=c$$
as required.

Next, we show that the rank of $M'$ is strictly less than $\min\{n, 2d-1\}$, and that any optimal partition will have at least $k$ parts.
Consider an optimal partition $\cP$ of $E$ so that $\#\cP$ is as small as possible. Let $i\in [k]$. By Lemma \ref{rank}, there exists a partition $\cQ$ of $E_i$ so that 
$\tilde{r}(\cP)+\tilde{r}(\{E_i\})\geq \tilde{r}(\cP\vee\{E_i\})+\tilde{r}(\cQ).$
Since $\{E_i\}$ is an optimal partition for $E_i$, if follows that $\cP\vee\{E_i\}$ is optimal partition for $E$. By our choice of $\cP$, we have $\#\cP\leq \#(\cP\vee\{E_i\})$, so that $E_i\subseteq P$ for some $P\in \cP$. It follows that $\cP$ is a coarsening of $\cP_0$. If $r(P)<d$ for each $P\in \cP$, then 
$\tilde{r}(\cP)=2ck-\#\cP$, so that $\cP=\cP_0$ by the optimality of $\cP$. If on the other hand $r(P)=d$ for some part $P\in \cP$, then 
$\tilde{r}(\cP)\geq 2d-1=2ck-2c-1>2ck-k=\tilde{r}(\cP_0)$ as $k>2c+1$, contradicting that $\cP$ is optimal. We conclude that $\cP=\cP_0$, and in particular, that there are no optimal partitions with less than $\#\cP_0=k$ parts. Then the rank of $M'$ is $\tilde{r}(\cP_0)=2ck-k<\min\{2ck, 2(ck-c)-1\}=\min\{n, 2d-1\}$, as required.

Finally, as $U_{c,2c}$ is representable over $\C$ and taking disjoint
unions and truncation preserves representability over $\C$, $M$
is also the matroid of some subspace $V \subseteq \CC^{n}$,
and $M'$ the matroid of the amoeba $\cA(V \cap
(\CC^*)^{n})$ (by Theorem~\ref{thm:Main} and
Proposition~\ref{prop:AmoebaMatroid}).
\end{example}

\section{An algorithm for evaluating $r'$}
\label{sec:Algorithm}

We retain the notation of the previous section and continue to explore
the functions $r'$ and $\tilde{r}$ derived from a fixed rank function $r$
of a loopless matroid $M$ on a finite set $E$, aiming for an algorithm
to evaluate $r'(S)$ given any $S\subseteq E$.

\subsection{Optimal partitions form a lattice} 
\begin{lm} \label{lattice}Let $\cP, \cP'$ be optimal
partitions of $S, S'\subseteq E$, respectively. 
If $$r'(S)+r'(S')=r'(S\cup S')+r'(S\cap S'),$$ then  $\cP\vee \cP'$ is optimal for $S\cup S'$ and $\cP\wedge \cP'$ is optimal for $S\cap S'$.
\end{lm}
\proof
By Lemma \ref{rank}, there exists a partition $\cQ$ of $S\cap S'$ that is a coarsening of the partition $\cP\wedge \cP'$ so that 
$$r'(S)+r'(S')=\tilde{r}(\cP)+\tilde{r}(\cP')\geq \tilde{r}(\cP\vee \cP')+ \tilde{r}(\cQ)\geq r'(S\cup S')+r'(S\cap S').$$
By our assumption that  $r'(S)+r'(S')=r'(S\cup S')+r'(S\cap
S')$, we have $r'(S\cup S')=\tilde{r}(\cP\vee \cP')$ and
$r'(S\cap S')=\tilde{r}(\cQ)$. So $\cP \vee \cP'$ is optimal
for $S \cap S'$, and $\cQ$ is optimal for $S \cap S'$ as
well as a coarsening of $\cP \wedge \cP'$. 

Let $\cQ^*$ be an optimal partition of $S \cap S'$, as well
as a coarsening of  $\cP\wedge \cP'$, with $\#\cQ^*$ as
large as possible. By Lemma \ref{rank} there is a coarsening
$\cQ'$ of $\cP\wedge \cQ^*$ with 
$$r'(S)+r'(S \cap S')=\tilde{r}(\cP)+\tilde{r}(\cQ^*) \geq
\tilde{r}(\cP\vee \cQ^*)+ \tilde{r}(\cQ') \geq r'(S) + r'(S
\cap S'),$$
so that we again have equality throughout, 
and with $\#\cP+\#\cQ^*=\#(\cP\vee \cQ^*)+ \#\cQ'$. Then $\cQ'$ is also a coarsening of  $\cP\wedge \cP'$ and $r'(S\cap S')=\tilde{r}(\cQ')$, so that  $\#\cQ'\leq \#\cQ^*$ by our choice of $\cQ^*$. It follows that $\#\cP\geq \#(\cP\vee \cQ^*)$
and hence each element of $\cQ^*$ is a subset of an element
of $\cP$. Similarly, each element of $\cQ^*$ is the subset
of an element of $\cP'$. Then $\cQ^*=\cP\wedge \cP'$, and so
$\cP \wedge \cP'$ is optimal for $S \cap S'$. 
\endproof

\subsection{Coarsest optimal partition and submodular
minimisation}

It follows from Lemma \ref{lattice} that for any set $S\subseteq E$ there is a unique coarsest optimal partition 
$$\cP^*:=\bigvee_{\cP\text{ optimal for }S}\cP$$ 
so that each optimal partition for $S$ refines $\cP^*$. Similarly, there is a unique finest optimal partition 
$$\cP_*:=\bigwedge_{\cP\text{ optimal for }S}\cP$$
which refines each optimal partition for $S$.

We will describe an algorithm to calculate the coarsest optimal partition for any given $S\subseteq E$. 
The coarsest optimal partition $\cP$ for $S$ equals
\begin{equation}\label{parts}\cQ:=\{Q\subseteq S \mid Q\text{ is inclusion-wise maximal so that } r'(Q)=2r(Q)-1\}\end{equation}
To see this, note that because $\cP$ is optimal,
each part $P\in \cP$ has $r'(P)=2r(P)-1$ and hence is contained in one of the elements $Q\in\cQ$.  The inclusion $P\subseteq Q$ cannot be strict, for then $\cP\vee\{Q\}$ would be a coarser optimal partition for $S$ than $\cP$ by Lemma \ref{lattice}. Hence $\cP=\cQ$.

Because \eqref{parts} refers to the value $r'(Q)$, which we
do not yet know how to compute, \eqref{parts} 
is not directly useful for finding the coarsest partition
for a general subset $S\subseteq E$. But if we assume that
$S$ is independent in $M'$, i.e., that $|S|=r'(S)$, then
each $Q\subseteq S$ has $r'(Q)=|Q|$, so that we can identify
the part $Q$ of the coarsest optimal partition containing a
given $e\in S$ as the largest $Q\subseteq S$ such that $e\in
Q$ and $2r(Q)-1=|Q|$. Finding this $Q$ can be cast as a
submodular function minimisation problem; in the following
lemma to this effect, $B+e$ plays the role of $S$.

\begin{lm}\label{test}Let $B\subseteq E$ and $e\in E\setminus B$ be such that $r'(B+e)=|B+e|$. Then the function $f:2^B\rightarrow \mathbb{Q}$ determined by
\begin{equation}\label{cunningham}f(I):=2r(I+e)-1 - |I+e|-\frac{|I|}{2|B|}\end{equation}
is submodular, and  $J$ is a largest subset of $B$ so that $2r(J+e)-1=|J+e|$ if and only $f(J)=\min_I f(I)$.
\end{lm}
\proof The submodularity of $f$ follows from the submodularity of the
rank function $r$. For each $I\subseteq B$ we have
$|I+e|=r'(I+e)\leq \tilde{r}(\{I+e\})=2r(I+e)-1$; hence,
$$ f(I)\leq 0 ~ \Leftrightarrow ~ 2r(I+e)-1\leq |I+e| +\frac{|I|}{2|B|} ~ \Leftrightarrow ~ 2r(I+e)-1= |I+e| ~\Leftrightarrow ~f(I) = -\frac{|I|}{2|B|}. $$
Here, the second $\Leftrightarrow$ uses the fact that $0\leq\frac{|I|}{2|B|}\leq\frac{1}{2}$, while every other term is an integer.
Since $f(\emptyset)=0$, we have $0\geq \min_I f(I)$. The lemma follows.\endproof

That a submodular function $f$ obtained from a matroid rank function as in \eqref{cunningham} can be minimised in polynomial time was first established by Cunningham in \cite{Cunningham1984}. The weakly polynomial time algorithm of Lee, Sidford, and Wong \cite{Lee2015} for submodular set function minimisation takes
 $O(k^2\log(k)\cdot\gamma+k^3\log^{O(1)}(k))$ time to minimise $f$, where $k:=|B|$ and $\gamma$ is the time needed to evaluate $f$.

\subsection{Proof of Theorem~\ref{thm:Algorithm}} To make sure that an optimal partition for $S$ is the coarsest, it will suffice to consider the intersection of that partition with a spanning subset $S'$ of $S$. 

\begin{lm} \label{extend} Suppose that $S'\subseteq S\subseteq E$, and let $\cP', \cP$ be partitions of $S', S$ resp. so that $\cP'=\cP\wedge \{S'\}$. Then:
\begin{enumerate}
\item $\tilde{r}(\cP)=r'(S)=r'(S')$ if and only if $r(P\cap S')=r(P)$ for all $P\in\cP$ and $\cP'$ is optimal for $S'$, and
\item if $\tilde{r}(\cP)=r'(S)=r'(S')$, then $\cP$ is the coarsest optimal partition of $S$ if and only if  $\cP'$ is the coarsest
optimal partition of $S'$.
\end{enumerate}
\end{lm}

\proof We have 
$$\tilde{r}(\cP')=\sum_{P\in
\cP, P\cap S'\neq \emptyset} (2r(P\cap S')-1)\leq \sum_{P\in
\cP} (2r(P)-1)= \tilde{r}(\cP),$$
so that  $\tilde{r}(\cP')=\tilde{r}(\cP)$ if and only if $r(P\cap S')=r(P)$ for all $P\in\cP$. Since $r(P)>0$ for all $P\in \cP$, it follows that if  $\tilde{r}(\cP')=\tilde{r}(\cP)$, then $\#\cP'=\#\cP$.    
\begin{enumerate}
\item
If $\tilde{r}(\cP)=r'(S)=r'(S')$, then since $r'(S')\leq \tilde{r}(\cP')\leq \tilde{r}(\cP)$ we have $\tilde{r}(\cP')=\tilde{r}(\cP)$ and $r'(S')= \tilde{r}(\cP')$. Conversely if $\tilde{r}(\cP')=\tilde{r}(\cP)$ and $r'(S')= \tilde{r}(\cP')$, then $r'(S)\geq r'(S)= \tilde{r}(\cP')=\tilde{r}(\cP)\geq r'(S)$, and then we have equality throughout.
\item 
Now assume that $\tilde{r}(\cP)=r'(S)=r'(S')$, so that
$\cP'$ is optimal for $S'$ and $\tilde{r}(\cP')=\tilde{r}(\cP)$. 
Suppose that $\cP$ is not the coarsest optimal partition of $S$ and that say, $\cQ$  is a coarser optimal partition of $S$. Then $\cQ':=\cQ\wedge\{S'\}$ 
is an optimal partition for $S'$ by (1) and $\#\cQ'=\#\cQ<\#\cP=\#\cP'$, so that $\cP'$ is not the coarsest optimal partition of $S'$. Conversely, suppose that $\cP'$ is not coarsest and that say, $\cQ'$ is an optimal partition of $S'$ coarser than $\cP'$. Then $\cQ:=\cP\vee\cQ'$ is an optimal partition of $S$ by Lemma \ref{lattice}. It follows that $\cQ\wedge\{S'\}=\cQ'$, and hence $\#\cQ= \#\cQ'<\#\cP'=\#\cP$.  Then $\cP$ is not the coarsest optimal partition of $S$.
\end{enumerate}
This proves the two parts of the lemma.
\endproof

\begin{algorithm}
\caption{\label{alg:matroid} Coarsest optimal partition $\cP$ and basis $B$ for $S\subseteq E$}
\begin{algorithmic}
\IF{$S=\emptyset$}
\STATE Put $\cP:= \emptyset$, $B:= \emptyset$ and return $\cP$, $B$
\ELSE 
\STATE Pick $e\in S$, and put $S':=S-e$
\STATE Compute the coarsest optimal partition $\cP'$ and a basis $B'$ for $S'$
\IF{$r(P'+e)=r(P')$ for a $P'\in \cP'$}
\STATE Put $\cP:=\cP'\vee\{P'+e\}, B:=B'$
\STATE return $\cP$, $B$
\ELSE
\STATE Compute a largest set $J\subseteq B'$ such that $2r(J+e)-1=|J+e|$
\STATE Put $\cP:= \cP'\vee\{J+e\}$, $B:= B'+e$
\STATE return $\cP$, $B$
\ENDIF 
\ENDIF
 \end{algorithmic}
\end{algorithm}

\begin{thm} Given a set $S\subseteq E$, Algorithm \ref{alg:matroid} determines the coarsest optimal partition 
$\cP$ for  $S$ and a subset $B\subseteq S$ such that $$|B|=r'(B)=r'(S)=\tilde{r}(\cP).$$ Moreover, the algorithm runs in polynomial time, taking $O(nk+ k^3\log(k))$ rank evaluations in $M$, where $n:=|S|, k:=r'(S)$.
\end{thm}
\proof We first argue that the output of the algorithm is correct, using induction on $|S|$. 
The case that $S=\emptyset$ is trivial, so assume that $S\neq \emptyset$, and let $e\in S$. By induction, the algorithm correctly computes the coarsest optimal partition $\cP'$ for $S':=S-e$ and a subset $B'\subseteq S'$ such that $|B'|=r'(B)=r'(S')=\tilde{r}(\cP')$ initially. 

If $r(P'+e)=r(P')$ for some $P'\in\cP'$, then consider the output $\cP:=\cP'\vee\{P'+e\}$ and $B:=B'$ in this case. By construction of $\cP$, we have $\cP'=\cP\vee\{S'\}$, and $r(P-e)=r(P)$ for all $P\in \cP$. By  application of Lemma \ref{extend}(1) to $S', S, \cP', \cP$, we find that $\tilde{r}(\cP)=r'(S)=r'(S')$. Since $\cP'$ is the coarsest optimal partition of $S'$, part (2) of the same lemma yields that $\cP$ is the coarsest optimal partition of $S$. Since the rank and cardinality of $B=B'$ equal $r'(S')=r'(S)$, the output is correct in this case.

If there is no $P'\in\cP'$ so that $r(P'+e)=r(P')$, then the algorithm proceeds to find some largest set $J\subseteq B'$ such that $2r(J+e)-1=|J+e|$. We will show that the output $\cP:= \cP'\vee\{J+e\}$, $B:= B'+e$ is again correct. We first argue that $r(S)=r'(S')+1$. Consider the coarsest optimal partition $\cQ$ of $S$, and put $\cQ':=\cQ\wedge\{S'\}$. If $r'(S)=r'(S')$, then an application of Lemma \ref{extend}(2) to $S', S, \cQ', \cQ$ shows that $\cQ'=\cP'$, and using part (1) we find that for the element $Q\in \cQ$ so that $e\in Q$, we have $r(Q-e)=r(Q)$. Then $P':=Q-e\in \cQ'=\cP'$ has $r(P'+e)=r(P')$, a contradiction. So $r(S)=r'(S')+1$. By submodularity of $r'$, we have 
$r'(S')+r'(B)\geq r'(S)+r'(B').$
It follows that $r'(B)=r(B')+1=|B'|+1=|B|$. Since $J+e\subseteq B$, we have 
$$r'(J+e)=|J+e|=2r(J+e)-1=\tilde{r}(\{J+e\})$$
so that $\{J+e\}$ is an optimal partition for $J+e$. By Lemma \ref{lattice}, and noting that
$$r'(S')+r'(J+e)=r(S)+r'(J),$$
if follows that $\cP= \cP'\vee\{J+e\}$ is optimal for $S$. Let $\cQ:=\cP\wedge\{B\}$. Then $\cQ$ is an optimal partition of $B$ by Lemma \ref{extend}(1) applied to $B, S, \cQ, \cP$. If there is a coarser optimal partition $\cQ'$ of $B$, then there is a part $Q'\in \cQ'$ that is not contained in any part of $\cQ$, hence is not contained in any part of $\cP$. If $e\not\in Q'$, then $\cP'\vee\{Q'\}$ is a coarser optimal partition of $S'$ than $\cP'$, a contradiction. So $e\in Q'$, hence $J+e$ is properly contained in $Q'$, and $|Q'|=\tilde{r}(\{Q'\})=2r(Q')-1$. But then $J':=Q'-e$ contradicts the choice of $J$. Therefore $\cQ$ is the coarsest optimal partition of $B$. By Lemma \ref{extend}(2) applied to $B, S, \cQ, \cP$, it follows that $\cP$ is the coarsest optimal partition of $S$.

It remains to show that the algorithm takes polynomial time,
and $O(nk+ k^3\log(k))$ rank evaluations in $M$. We only count the number of rank
evaluations in $M$, the remaining work clearly being less significant
in comparison. Not counting the rank evaluations used to recursively compute $B', \cP'$ for $S':=S-e$, the algorithm performs at most $\#\cP'\leq r'(S')\leq r'(S)=k$ rank evaluations to test if $r(P'+e)=r(P')$ for a $P'\in \cP'$. If so, then $|B|=|B'|$ and no further rank evaluations are performed. If not, then $|B|=|B'|+1$ and for the calculation of $J$ one may use submodular
function minimisation as in Lemma \ref{test}, taking $O(k^2\log(k))$
evaluations of the submodular function $f$, each evaluation of $f$ taking
one rank evaluation in $M$. The depth of the recursion equals $|S|$,
with at most $k$ rank evaluations at each depth for testing and
$O(k^2\log(k))$ for computing $J$ at each depth where the cardinality
of the basis $B$ increases. Thus the entire recursive algorithm will
see at most $|S|\cdot k + |B| \cdot O(k^2\log(k))= O(nk+k^3\log(k))$ rank evaluations in $M$, as required.
 \endproof
\subsection*{Remark} It was pointed out to us by an anonymous referee that the matroid $M'$ arises as a sub-matroid of the {\em Dilworth truncation} $N$ of the direct sum of two copies $M_1$ and $M_2$ of $M$. That is, $$r'(F)=r_N\left(\{\{f_1, f_2\}: f\in F\}\right),$$ where $e_i\in E(M_i)$ denotes the copy of $e\in E(M)$. See also \cite[Ch. 48]{SchrijverCO_B} for an algorithm to evaluate the rank function of a Dilworth truncation in the more general context of submodular functions. In that reference, 
$$r'(U)=\hat{f}(U):=\min\left\{\sum_{P\in \mathcal{P}} f(P): \mathcal{P} \text{ a partition of } U\right\}$$
where we choose $f(P):=2r(P)-1$ as the base submodular function.
\bibliographystyle{alpha}
\bibliography{math}

\begin{thebibliography}{{Stu}02}

\bibitem[Ber71]{Ber-LogarithmicLimitSet}
G.~M. Bergman.
\newblock The logarithmic limit-set of an algebraic variety.
\newblock {\em Trans. Amer. Math. Soc.}, 157:459--470, 1971.

\bibitem[Cun84]{Cunningham1984}
William~H. Cunningham.
\newblock Testing membership in matroid polyhedra.
\newblock {\em J. Combin. Theory Ser. B}, 36(2):161--188, 1984.

\bibitem[DRY20]{Draisma18d}
Jan Draisma, Johannes Rau, and Chi~Ho Yuen.
\newblock The dimension of an amoeba.
\newblock {\em Bull.~Lond.~Math.~Soc.}, 52(1):16--23, 2020.

\bibitem[Egg22]{Eggleston22}
Sarah Eggleston.
\newblock The dimension of amoebas of linear spaces.
\newblock Master's thesis, Faculty of Science, University of Bern, 2022.

\bibitem[Har92]{Harris1992}
Joe Harris.
\newblock {\em Algebraic geometry. {A} first course}, volume 133 of {\em Grad.
  Texts Math.}
\newblock Berlin etc.: Springer-Verlag, 1992.

\bibitem[LSW15]{Lee2015}
Yin~Tat Lee, Aaron Sidford, and Sam Chiu-wai Wong.
\newblock A faster cutting plane method and its implications for combinatorial
  and convex optimization.
\newblock In {\em 2015 {IEEE} 56th {A}nnual {S}ymposium on {F}oundations of
  {C}omputer {S}cience---{FOCS} 2015}, pages 1049--1065. IEEE Computer Soc.,
  Los Alamitos, CA, 2015.

\bibitem[NS22]{Nisse22}
Mounir Nisse and Frank Sottile.
\newblock Describing amoebas.
\newblock {\em Pac. J. Math.}, 317(1):187--205, 2022.

\bibitem[Rau20]{Rau20}
Johannes Rau.
\newblock The dimension of an amoeba.
\newblock LAGARTOS,\\
  \url{https://sites.google.com/site/cotterillethan/ethan-cotterill-eng/latin-american-real-and-tropical-geometry-seminar},
  2020.

\bibitem[Sch03]{SchrijverCO_B}
Alexander Schrijver.
\newblock {\em Combinatorial optimization. {P}olyhedra and efficiency. {V}ol.
  {B}}, volume 24,B of {\em Algorithms and Combinatorics}.
\newblock Springer-Verlag, Berlin, 2003.
\newblock Matroids, trees, stable sets, Chapters 39--69.

\bibitem[{Stu}02]{Stu-SolvingSystemsPolynomial}
Bernd {Sturmfels}.
\newblock {\em Solving systems of polynomial equations}.
\newblock Number~97 in CBMS Regional Conferences Series. Providence, RI:
  American Mathematical Society, 2002.

\end{thebibliography}

\end{document}